\documentclass[11pt]{amsart}

\usepackage[T1]{fontenc}
\usepackage[utf8]{inputenc}
\usepackage{lmodern}
\usepackage{amsmath,amssymb,amsthm,mathtools}
\usepackage{microtype}
\usepackage{enumitem}
\usepackage{booktabs}
\usepackage{graphicx}
\usepackage{hyperref}
\usepackage[margin=1.03in]{geometry}
\usepackage{tikz}
\usetikzlibrary{arrows.meta,calc,angles,quotes,decorations.pathreplacing}

\hypersetup{colorlinks=true,linkcolor=blue,citecolor=blue,urlcolor=blue}

\newtheorem{theorem}{Theorem}[section]
\newtheorem{lemma}[theorem]{Lemma}
\newtheorem{proposition}[theorem]{Proposition}
\newtheorem{corollary}[theorem]{Corollary}
\theoremstyle{definition}
\newtheorem{definition}[theorem]{Definition}

\newtheorem{problem}[theorem]{Problem}
\theoremstyle{remark}

\newcommand{\R}{\mathbb R}
\newcommand{\RP}{\mathbb {RP}}
\newcommand{\dd}{\,\mathrm d}
\newcommand{\sgn}{\operatorname{sgn}}

\newcommand{\Top}{\operatorname{Top}}

\title{Counterexamples and Symmetry for Uneven Orthogonal Mass Partitions in the Plane}
\author{Leonardo Mart\'inez-Sandoval}
\address{Faculty of Sciences, UNAM, Mexico City, Mexico}
\email{leomtz@ciencias.unam.mx}
\subjclass[2020]{Primary 52A37; Secondary 52A38, 52C35, 52-08.}
\keywords{mass partitions, orthogonal partitions, Grünbaum's conjecture, affine symmetry, Gaussian perturbations, discrete geometry}

\begin{document}

\begin{abstract}
Gr\"unbaum asked whether every planar convex body admits, for every $0\leq t\leq 1/4$, two orthogonal lines cutting it into pieces with cyclically ordered areas $t,t,1/2-t,1/2-t$.  B\'ar\'any posed the analogous question for well-behaved planar measures and conjectured that the answer there is negative.

We confirm B\'ar\'any's conjecture in a particularly robust form: for every fixed $0<t<1/4$ we construct smooth, strictly positive, centrally symmetric, strongly log-concave measures arbitrarily close to the standard Gaussian for which the prescribed partition does not exist.  In contrast, we prove that the partition exists for every $t$ whenever the measure is invariant under an orientation-reversing affine involution. We also exhibit a $96$-point counterexample for which no pair of perpendicular lines produces cyclic counts $8,8,40,40$.

\end{abstract}

\maketitle

\section{Introduction}

\subsection{Literature overview}
Mass partition problems ask whether prescribed proportions of one or several measures can be realized by geometrically constrained cuts.  The constraints may concern the cutting objects themselves or their mutual position, such as prescribed incidences, intersections, or orthogonality.  The subject has a long tradition at the interface of convex and discrete geometry, topology, and computational geometry; see the survey of Rold\'an-Pensado and Sober\'on \cite{RoldanSoberon2022}.

A central early question is a problem posed by Gr\"unbaum in 1960: can every sufficiently regular mass in $\R^d$ be cut into $2^d$ equal parts by $d$ affine hyperplanes?  Gr\"unbaum noted that the answer is positive for $d\le2$ \cite{Grunbaum1960}; Hadwiger proved the case $d=3$ \cite{Hadwiger1966}; and Avis constructed counterexamples for every $d\ge5$ \cite{Avis1984}.  Sober\'on recently settled the remaining case $d=4$ negatively by constructing a smooth strictly positive density which cannot be equipartitioned by four affine hyperplanes \cite{Soberon2026}. For the history and earlier progress on Gr\"unbaum's problem, see Blagojevi\'c, Frick, Haase and Ziegler \cite{BlagojevicFrickHaaseZiegler2018}.

Gr\"unbaum also observed in the same paper that every continuous planar mass admits an equipartition by two perpendicular lines \cite{Grunbaum1960}.  Requiring mutual orthogonality in higher dimensions is much more rigid: Maldonado and Rold\'an-Pensado proved that for $d\geq 3$ there are masses in $\R^d$ which cannot be equipartitioned by $d$ mutually orthogonal hyperplanes \cite{MaldonadoRoldan2025}, thereby completing the problem under the orthogonality restriction together with the results above. A related complementary positive result was recently proved by Musin \cite{Musin2026}, who showed that suitable measures on $\mathbb{R}^d$ admit $d$ mutually orthogonal hyperplanes such that every pair divides the measure into four equal parts.

The present paper concerns an uneven planar version.  As the contribution ``Quadrupartitions'' for the \emph{Open Problems} chapter in \emph{Fete of Combinatorics and Computer Science}, B\'ar\'any attributes the following problem to Gr\"unbaum \cite{Alon2010OpenProblems}.

\begin{problem}[Gr\"unbaum's convex-body problem]\label{q:grunbaum}
If $K\subset\R^2$ has area one and $0\leq t\le1/4$, must there be two orthogonal lines whose four regions, in cyclic order, have areas 
\[
 t,\quad t,\quad \frac12-t,\quad \frac12-t ?
\]
\end{problem}

In the same text, B\'ar\'any also asks the corresponding question for line-null probability measures.

\begin{problem}[B\'ar\'any's mass-partition problem]\label{q:barany}
If $\mu$ is a Borel probability measure on $\R^2$, assigning measure zero to every affine line, and $0\leq t\leq 1/4$, must there be two orthogonal lines whose four regions, in cyclic order, have masses
\[
 t,\quad t,\quad \frac12-t,\quad \frac12-t ?
\]
\end{problem}

The two questions were expected to behave differently.  Gr\"unbaum's convex-body problem remains open and is conjectured to have a positive answer, whereas B\'ar\'any conjectured that the analogous statement for measures is false.

A notable partial result for convex bodies is due to Arocha, Jer\'onimo-Castro, Montejano and Rold\'an-Pensado: they proved the conjecture when the diameter is at least $\sqrt{37}$ times the minimum width, with the improved constant $3$ for centrally symmetric bodies \cite{ArochaEtAl2010}. A related spherical version was studied by Blagojević and Dimitrijević Blagojević \cite{BlagojevicDimitrijevic2013}.

Symmetry assumptions are especially relevant here because in several neighboring partition problems they restore or strengthen positive results.  Makeev proved that if two continuous mass distributions in $\R^n$ have a common center of symmetry, then there are $n$ hyperplanes through that center such that any two of them quarter both masses \cite{Makeev2007}.  More recently, Fradelizi, Hubard, Meyer, Rold\'an-Pensado and Zvavitch showed that every centrally symmetric convex body in $\R^3$ admits three planes through the origin which split the body into eight equal-volume pieces and simultaneously quarter each central planar section by the other two planes \cite{FradeliziEtAl2022}.  These results suggest asking which kinds of symmetry might force a positive answer to Problems~\ref{q:grunbaum} and~\ref{q:barany}.

The problem also has a finite version. Following the terminology of Maldonado's thesis \cite{MaldonadoThesis2025}, one may ask the following.

\begin{problem}\label{q:discrete}
Let $P$ be a finite set of $N=2m$ points in $\R^2$, and let $0\le k\le N/4$ be an integer.  Must $P$ admit an orthogonal $k$-partition, that is, a partition into four cyclically ordered classes of sizes
\[
 k,\quad k,\quad m-k,\quad m-k
\]
separated by two perpendicular lines?
\end{problem}

Maldonado and Rold\'an-Pensado \cite{MaldonadoThesis2025} proved that every set of an even number of planar points admits an orthogonal $1$-partition, and that every such point set in convex position admits an orthogonal $k$-partition for every $k$. They also report substantial computational searches for a counterexample: variants of the algorithms developed there, together with differential evolution as a search heuristic, found none among configurations with up to $26$ points.

\subsection{Our results}
Our first result settles B\'ar\'any's question in a robust form.

\begin{theorem}\label{thm:counterintro}
For every $0<t<1/4$ there exists a smooth, strictly positive, centrally symmetric, strongly log-concave probability density on $\R^2$ (hence, a line-null probability measure) for which no two perpendicular lines cut the measure cyclically into masses
\[
 t,\quad t,\quad \frac12-t,\quad \frac12-t.
\]
The density may be chosen arbitrarily close in $C_2(\mathbb{R}^2)$ to the standard Gaussian.
\end{theorem}

Motivated by Sober\'on's Gaussian construction in dimension four \cite{Soberon2026}, we find counterexamples among small perturbations of the standard planar Gaussian. A first-order analysis shows that whenever the linear term is nonzero, the desired partitions still exist for sufficiently small perturbations (Proposition~\ref{prop:linearized-defect}). We therefore choose a perturbation for which that term vanishes and use a second-order calculation to rule out the desired partition in every direction (Theorem~\ref{thm:explicit-counterexample}).

The negative theorem does not mean that symmetry is irrelevant.  Our second result shows that a particular kind of symmetry restores the partition for all targets.

\begin{theorem}\label{thm:mainintro}
Let $\mu$ be a Borel probability measure on $\R^2$ assigning measure zero to every affine line.  Suppose that $\mu$ is invariant under an orientation-reversing affine involution.  Then, for every $0\leq t\leq 1/4$, there exist two orthogonal lines whose four regions, in cyclic order, have masses
\[
 t,\quad t,\quad\frac12-t,\quad\frac12-t.
\]
\end{theorem}

The contrast with Theorem~\ref{thm:counterintro} is that central inversion preserves orientation, whereas the symmetry in Theorem~\ref{thm:mainintro} reverses it.  Starting from a measure and one direction, we choose a second, not necessarily perpendicular direction so that the four region masses have the desired values.  As the first direction turns, the second direction varies continuously (Theorem \ref{thm:pairedmap}).  We then introduce a signed integral for this variation: an orientation-reversing affine symmetry changes its sign while preserving its magnitude (Proposition \ref{prop:affineJ}), forcing the integral to vanish and hence forcing the desired pair of perpendicular directions. For the strategy to work, we must first work with regular measures (see Section \ref{sec:monotonicity}) and then use them to approximate line-null measures (see Appendix \ref{app:line-null}).

This has direct consequences for Gr\"unbaum's problem for convex bodies.  Section~\ref{sec:convex} shows, in particular, that every convex body with an orientation-reversing affine automorphism satisfies the conjecture for every target (Corollary \ref{cor:body}). This includes all triangles (Corollary \ref{cor:triangle}) and all trapezoids (Corollary \ref{cor:trapezoid}). 

A common elementary issue underlies both main proofs: a boundary line is defined by a fixed-mass condition and therefore moves when either the measure is perturbed or the line direction is rotated.  Section~\ref{sec:variation} isolates the implicit-differentiation rule that tracks this motion, so that the later arguments can focus on the geometry it produces.

The finite problem also has a negative answer in general.  Standard approximation arguments connect continuous mass-partition statements with finite point-set versions under suitable boundary conventions; see Rold\'an-Pensado--Sober\'on \cite{RoldanSoberon2022} and, for the general discrete-geometric viewpoint, Matou\v{s}ek \cite{Matousek2002}.  Thus one can in principle discretize the continuous counterexample above.  A naive quantitative implementation, however, must approximate the relevant sector masses uniformly over a large family of pairs of lines and leads to an impractically large point set.  We instead search directly for a finite configuration.

\begin{theorem}\label{thm:discrete-counterexample}
There is a centrally symmetric set $P$ of $96$ integer points in $\R^2$, with no three collinear and no orthogonal $8$-partition. This remains true under the convention in which points on either cutting line may be assigned arbitrarily to adjacent sectors.
\end{theorem}

\begin{figure}[ht]
\centering
\includegraphics[width=0.7\textwidth]{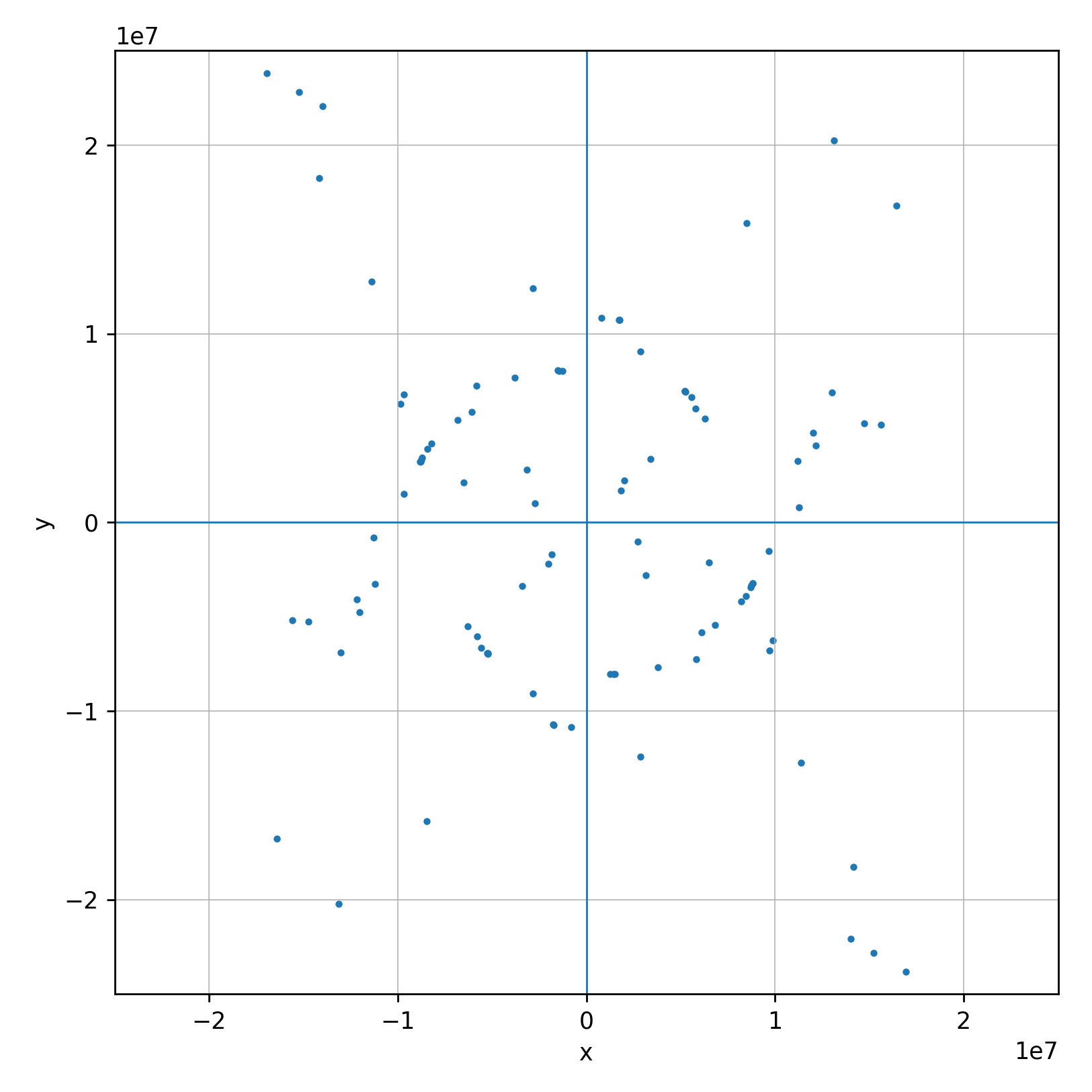}
\caption{A centrally symmetric $96$-point counterexample for $k=8$.}
\label{fig:discrete-counterexample}
\end{figure}

The example is shown in Figure \ref{fig:discrete-counterexample}, and it was obtained by a simulated-annealing search, documented in Appendix \ref{app:discrete-counterexample}. The code used for the simulated-annealing search and the subsequent exact verification is available at the GitHub repository \cite{quadrupartitions2026}.

Finally, the reversing-symmetry theorem gives a complementary positive discrete class.  Because the result is obtained by smoothing atomic measures and taking a limit, the resulting statement leaves points on the cutting lines unassigned to the open sectors (see Appendix \ref{app:line-null}).

\begin{corollary}\label{cor:discrete}
Let $P$ be a finite point multiset in $\R^2$, invariant as a multiset under an orientation-reversing affine involution, and write $n=|P|$.  For every integer $k$ with
\[
 0\le k\le\left\lfloor\frac n4\right\rfloor,
\]
there are two perpendicular lines whose four open sectors contain, in cyclic order, at most
\[
 k,\quad k,\quad
 \left\lfloor\frac{n-2k}{2}\right\rfloor,\quad
 \left\lfloor\frac{n-2k}{2}\right\rfloor
\]
points of $P$.
\end{corollary}

\subsection{Organization}
Section~\ref{sec:variation} fixes the common notation and isolates the implicit-variation calculation used throughout the paper.  Section~\ref{sec:first-order} explains why a nonzero first-order perturbation cannot give counterexamples arbitrarily close to the Gaussian.  Section~\ref{sec:counterexample} constructs the counterexamples by passing to second order.  Section~\ref{sec:monotonicity} studies the full zero set of the same imbalance and constructs the uniquely determined paired direction.  Section~\ref{sec:chirality} introduces the partition action and proves its affine sign law and the reversing-symmetry theorem for regular measures.  Section~\ref{sec:convex} applies that theorem to Gr\"unbaum's problem for convex bodies.  Section~\ref{sec:discussion} compares the two mechanisms with earlier approaches, explains the Liouville interpretation of the action, and records the main remaining questions.  Appendix~\ref{app:line-null} contains the approximation arguments; Appendix~\ref{app:discrete-counterexample} gives the coordinates of the finite example and its exact certificate.

\section{Notation, background conventions, and variation}\label{sec:variation}
Throughout the paper, for $t\in [0,1/4]$ it is convenient to use
\[
 p=2t\in[0,1/2].
\]

At $t=0$, we allow the degenerate partition consisting of a single halving line, and the case $t=1/4$ was settled by Gr\"unbaum in  \cite{Grunbaum1960}. Thus we may assume \(0<t< 1/4\).

For $u=(u_1,u_2)\in S^1$, write $u^\perp=(-u_2,u_1)$ and for $\theta \in [0,2\pi)$, write $u_\theta=(\cos\theta,\sin\theta)$. Let $\mathcal R_\varphi$ denote counterclockwise rotation through $\varphi$.  Thus $u_\theta^\perp=\mathcal R_{\pi/2} (u_\theta)=u_{\theta+\pi/2}$. For some harmonic calculations we will need $c_k(\theta)=(\cos k\theta,-\sin k\theta)^\top$.

Let $\mu$ have a strictly positive density. For $u,v\in S^1$,
let the \textit{thresholds} $m(u)$ and $q_p(v)$ be uniquely
determined by
\[
 \mu\{x:\langle x,u\rangle\ge m(u)\}=\frac12,
 \qquad
 \mu\{x:\langle x,v\rangle\ge q_p(v)\}=p.
\]
Define the \textit{median half-plane} and the \textit{$p$-cap} by
\[
 H_u^+=\{x:\langle x,u\rangle\ge m(u)\},
 \qquad
 C_{v,p}=\{x:\langle x,v\rangle\ge q_p(v)\}.
\]
Our basic two-direction \textit{imbalance} is
\begin{equation}\label{eq:master-imbalance}
 \Phi_{\mu,p}(u,v):=\mu(H_u^+\cap C_{v,p})-\frac p2.
\end{equation}
Thus $\Phi_{\mu,p}(u,v)=0$ precisely when the median with normal $u$ bisects the chosen $p$-cap with normal $v$ i.e. when the regions have the desired measures cyclically.  The \textit{orthogonality defect} is the restriction
\[
 \Delta_{\mu,p}(u):=\Phi_{\mu,p}(u,u^\perp).
\]
A zero of $\Delta_{\mu,p}$ is exactly a perpendicular partition with cyclic masses $p/2,p/2,(1-p)/2,(1-p)/2$. Conversely every such partition is represented by one of the two choices of side for the $p$-cap.

For calculations we write
\[
 U_\theta=\langle(x,y),u_\theta\rangle,\qquad
 V_\theta=\langle(x,y),u_\theta^\perp\rangle,
\]
and abbreviate $a_\theta=m(u_\theta)$ and $b_\theta=q_p(u_\theta^\perp)$.  Then
\[
 \Delta_{\mu,p}(\theta)
 =\mu\{U_\theta\ge a_\theta,\ V_\theta\ge b_\theta\}-\frac p2.
\]
When the density is written in the $(U_\theta,V_\theta)$ coordinates, this is an ordinary iterated integral over the moving quadrant $[a_\theta,\infty)\times[b_\theta,\infty)$.

As we perturb the measure or rotate the cutting directions,
the defect may be difficult to evaluate exactly, but its
derivatives can be computed by implicitly differentiating
the equations that determine the thresholds. So the following principle will be used repeatedly.
Suppose $\mathcal{A}(s,q)$ measures a certain region in terms of a parameter $s$ and a scalar threshold $q$. Then $q=q(s)$ is determined implicitly
by a fixing
\[
 \mathcal A(s,q(s))=p.
\]
When $\mathcal A_q\ne0$, the implicit-function theorem gives
\begin{equation}\label{eq:variation-constraint}
 q'(s)=-\frac{\mathcal A_s}{\mathcal A_q}.
\end{equation}
Then, for a quantity of interest $\mathcal B(s,q(s))$, the chain rule gives
\begin{equation}\label{eq:variation-chain}
 \frac{\dd}{\dd s}\mathcal B(s,q(s))
 =\mathcal B_s-\frac{\mathcal B_q\mathcal A_s}{\mathcal A_q}.
\end{equation}
In Sections~\ref{sec:first-order} and~\ref{sec:counterexample}, the parameter $s$ is a perturbation size and the moving thresholds are determined by fixing mass in certain regions of the perturbed density.  In Section~\ref{sec:monotonicity}, $s$ is a rotation angle and the moving threshold translates the rotating cap boundary so that its mass remains $p$.  This common ``differentiate the constraint, solve for the threshold motion, substitute into the desired derivative'' viewpoint is the key variation principle behind both halves of the paper.

We use standard calculus. For the probability background behind the Gaussian and threshold calculations in Sections~\ref{sec:first-order} and~\ref{sec:counterexample}, see Billingsley \cite{Billingsley1999} and van der Vaart \cite{vanDerVaart1998}; the Gaussian/Hermite identities used there may also be organized as in Thangavelu \cite{Thangavelu1993}.  The particular identities needed in the proof are recalled when they arise.

\section{Pure first-order perturbations cannot work}\label{sec:first-order}

Let $\mu_0$ be the standard planar Gaussian and write
\[
 \phi(s)=\frac{e^{-s^2/2}}{\sqrt{2\pi}},\qquad
 G(s)=\int_s^\infty\phi(v)\,\dd v.
\]
Consider $\dd\mu_\varepsilon=(1+\varepsilon h)\dd\mu_0$, where $h$ is smooth, bounded, and has zero $\mu_0$-mean.  For fixed $p$, choose $q$ with $G(q)=p$.

\begin{proposition}\label{prop:linearized-defect}
Uniformly in $\theta$,
\[
 \Delta_{\mu_\varepsilon,p}(\theta)=\varepsilon L_ph(\theta)+O(\varepsilon^2),
\]
where
\begin{equation}\label{eq:neg-linear}
 L_ph(\theta)=\mathbb E_{\mu_0}\!\left[
 h(X)(\mathbf 1_{U_\theta\ge0}-\tfrac12)
       (\mathbf 1_{V_\theta\ge q}-p)\right].
\end{equation}
Moreover
\[
 \int_0^{2\pi}L_ph(\theta)\,\dd\theta=0.
\]
\end{proposition}

\begin{proof}
Fix $\theta$ and abbreviate $U=U_\theta$, $V=V_\theta$.  Let $a_\varepsilon$ and $b_\varepsilon$ be the perturbed median and cap thresholds, so that
\[
 \mu_\varepsilon\{U\ge a_\varepsilon\}=\frac12,
 \qquad
 \mu_\varepsilon\{V\ge b_\varepsilon\}=p,
 \qquad
 a_0=0,\quad b_0=q.
\]
Put
\[
 A=\int_{U\ge0}h\,\dd\mu_0,\qquad
 B=\int_{V\ge q}h\,\dd\mu_0,\qquad
 W=\int_{U\ge0,V\ge q}h\,\dd\mu_0.
\]
Differentiating the two fixed-mass equations at $\varepsilon=0$ gives
\begin{equation}\label{eq:quantile-speeds}
 a_0'=\frac{A}{\phi(0)},\qquad
 b_0'=\frac{B}{\phi(q)}.
\end{equation}

Now let $F(a,b,\varepsilon)
 =\mu_\varepsilon\{U\ge a,V\ge b\}.$ The defect is $F(a_\varepsilon,b_\varepsilon,\varepsilon)-p/2$.  At $(a,b,\varepsilon)=(0,q,0)$,
\[
 F_\varepsilon=W,\qquad
 F_a=-p\phi(0),\qquad
 F_b=-\frac12\phi(q).
\]
Hence the chain rule and \eqref{eq:quantile-speeds} give
\[
 \left.\frac{\dd}{\dd\varepsilon}
 F(a_\varepsilon,b_\varepsilon,\varepsilon)\right|_{\varepsilon=0}
 =W-pA-\frac B2,
\]
which is exactly \eqref{eq:neg-linear} after expanding the indicator functions and using $\int h\,\dd\mu_0=0$. The implicit function theorem gives smooth dependence on $\varepsilon$ and $\theta$, and compactness of
the circle of directions makes the Taylor remainder uniform
in $\theta$.

For the mean zero statement, use \eqref{eq:neg-linear} directly. For each fixed \(x\), the four possible values of

$$ \bigl(\mathbf 1_{U_\theta(x)\ge0}-\tfrac12\bigr) \bigl(\mathbf 1_{V_\theta(x)\ge q}-p\bigr) $$
are
$$ \frac{1-p}{2},\quad -\frac p2,\quad \frac{p-1}{2},\quad \frac p2, $$
according to the four regions determined by \(U_\theta(x)=0\) and \(V_\theta(x)=q\). As \(\theta\) runs once around the circle, the angles for which \(V_\theta(x)\ge q\) are divided equally between \(U_\theta(x)\ge0\) and \(U_\theta(x)<0\); the same is therefore true when \(V_\theta(x)<q\). Hence the two opposite contributions cancel in each case, and
$$ \int_0^{2\pi} \bigl(\mathbf 1_{U_\theta(x)\ge0}-\tfrac12\bigr) \bigl(\mathbf 1_{V_\theta(x)\ge q}-p\bigr)\,\dd\theta=0. $$

Integrating in \(x\) and applying Fubini gives the desired equality.

\end{proof}

Thus any nonzero first-order orthogonality defect changes sign: if $L_ph\not\equiv0$, its zero angular mean gives directions $\theta_+$ and $\theta_-$ with $L_ph(\theta_+)>0$ and $L_ph(\theta_-)<0$.  For sufficiently small $\varepsilon$, the full defect has the same respective signs at these two directions, and hence vanishes for some orientation between them.  Therefore, to construct counterexamples for arbitrarily small $\varepsilon$, we must select $h$ so that $L_ph\equiv0$ and obtain the obstruction from higher-order terms.

\section{The negative result: a family of counterexamples}\label{sec:counterexample}

Now we focus on constructing the counterexamples. Fix $0<p<1/2$, let $q>0$ be determined by $G(q)=p$, and put $r:=x^2+y^2$.  Define
\[
 g_0(x,y)=\frac1{2\pi}e^{-r^2/2},\qquad
 g_1(x,y)=\frac1\pi e^{-r^2}.
\]
The narrower Gaussian $g_1$ will be useful because
\[
 \frac{g_1}{g_0}=2e^{-r^2/2};
\]
therefore a polynomial times $g_1$, divided by $g_0$, is a polynomial times $e^{-r^2/2}$ and is bounded together with the finitely many derivatives needed below.

Consider the following polynomials in $\sigma$:
\begin{align*}
 A_q(\sigma)&=2(\sigma-q^2)^2-8(\sigma-q^2)+4,\\
 B_q(\sigma)&=2(\sigma-q^2)^2-9(\sigma-q^2)+6.
\end{align*}
As we will see below, they are chosen so that the degree-two perturbation is invisible to the orthogonality defect at first order at the unperturbed threshold $b=q$, while it still moves the cap threshold. Differentiating the hidden defect in that moved direction then produces the desired second-order term.  Let
\begin{align*}
 \kappa_q&=\sqrt2\,e^{-q^2/2}\\
 P_{2,q}(x,y)
 &= (x^2-y^2)A_q(r^2)-2xyB_q(r^2),\\
 P_{4,q}(x,y)&=\frac{\kappa_q}{2}xy(x^2-y^2).
\end{align*}
These polynomials have a simple angular structure. In polar
coordinates $(x,y)=(r\cos\alpha,r\sin\alpha)$,
\[
 P_{2,q}
 =r^2\bigl(A_q(r^2)\cos2\alpha-B_q(r^2)\sin2\alpha\bigr),
 \qquad
 P_{4,q}=\frac{\kappa_q r^4}{8}\sin4\alpha.
\]
Thus $P_{2,q}$ is an angular harmonic of order two: on each
circle centered at the origin, its angular dependence is a
linear combination of $\cos2\alpha$ and $\sin2\alpha$.
Analogously, $P_{4,q}$ has angular frequency four.
Rotating the coordinates through $\theta$ therefore introduces
coefficients involving $2\theta$ and $4\theta$, respectively,
which will simplify the calculation of the defect.

Finally, for $\varepsilon>0$ set
\begin{equation}\label{eq:neg-density}
 \rho_{\varepsilon,q}
 =g_0+\varepsilon P_{2,q}g_1+\varepsilon^2P_{4,q}g_1,
 \qquad
 \dd\mu_{\varepsilon,q}=\rho_{\varepsilon,q}\,\dd x\,\dd y.
\end{equation}

\begin{theorem}[Gaussian counterexamples at every target]\label{thm:explicit-counterexample}
For every fixed $q>0$ there exists $\varepsilon_0(q)>0$ such that, for $0<\varepsilon<\varepsilon_0(q)$, \eqref{eq:neg-density} is a smooth, strictly positive, centrally symmetric, strongly log-concave probability density.  Moreover, uniformly in $\theta$,
\begin{equation}\label{eq:neg-asymptotic}
 \Delta_{\mu_{\varepsilon,q},p}(\theta)
 =\frac{K_q}{8}\varepsilon^2+O_q(\varepsilon^3),
 \qquad
 K_q=\frac{\sqrt2}{\pi}q^2e^{-3q^2/2}>0.
\end{equation}
Consequently, after decreasing $\varepsilon_0(q)$ if necessary,
\begin{equation}\label{eq:neg-gap}
 \Delta_{\mu_{\varepsilon,q},p}(\theta)>0
 \qquad\text{for every $\theta$ and }0<\varepsilon<\varepsilon_0(q),
\end{equation}
so no perpendicular pair realizes the cyclic masses $p/2,p/2,(1-p)/2,(1-p)/2$.
\end{theorem}

\begin{proof}
Central symmetry is immediate from the definitions, so every median line passes through the origin.  We compute the cap motion and the defect using the variation scheme of Section~\ref{sec:variation}.

Let
\[
 T_b=\frac{e^{-b^2}}{\sqrt\pi},\qquad
 D_b=\frac{e^{-b^2}}\pi,
 \qquad e_1=\begin{pmatrix}1\\0\end{pmatrix},
\]
and put $z=b^2-q^2$. The angular structure of the perturbing polynomials allows us to  define vectors $\mathbf B_{1,q}(b)$ and $\mathbf D_{1,q}(b)$ by
\[
 \int_{V\ge b} P_{2,q}\!\left(\mathcal R_\theta(U,V)\right)g_1(U,V)\,\dd U\,\dd V
 =\mathbf c_2(\theta)^T\mathbf B_{1,q}(b)
\]
and
\[
\begin{aligned}
 &\int_{U\ge0,V\ge b}P_{2,q}\!\left(\mathcal R_\theta(U,V)\right)g_1(U,V)\,\dd U\,\dd V\\
 &\qquad-\frac12\int_{V\ge b}P_{2,q}\!\left(\mathcal R_\theta(U,V)\right)g_1(U,V)\,\dd U\,\dd V
 =\mathbf c_2(\theta)^T\mathbf D_{1,q}(b).
\end{aligned}
\]
Define $\mathbf B_{2,q}$ and $\mathbf D_{2,q}$ analogously for the degree-four term.  The first pair records the cap-mass contribution and the corresponding half-cap defect after the angular factor has been extracted. The second pair does the same for the fourth harmonic.

Under $g_1$, the coordinates $U$ and $V$ are independent centered Gaussians of variance $1/2$.  Since the cap $\{V\ge b\}$ contains the entire $U$-axis, terms odd in $U$ integrate to zero, and the remaining terms factor into ordinary even Gaussian moments in $U$ and one-sided truncated moments in $V$. We can then find closed formulas recursively by integration by parts (see alternatively \cite[\S1.1, equations (1.1.1)--(1.1.4)]{Thangavelu1993}).

For the degree-two perturbation one obtains
\begin{align}
 \mathbf B_{1,q}(b)
 &=-\frac{bT_b}{2}
 \begin{pmatrix}
  2z^2-2z-\frac12\\
  z(2z-3)
 \end{pmatrix},\label{eq:neg-B1}\\
 \mathbf D_{1,q}(b)
 &=\frac{D_b}{2}z
 \begin{pmatrix}
  1-2z\\2z
 \end{pmatrix}.\label{eq:neg-D1}
\end{align}
For the degree-four term one has
\begin{align*}
 \mathbf B_{2,q}(b)
 &=-\frac{\kappa_q}{16}bT_b\left(b^2-\frac32\right)
 \begin{pmatrix}0\\1\end{pmatrix},\\
 \mathbf D_{2,q}(b)
 &=-\frac{\kappa_q}{8}D_bb^2e_1.
\end{align*}

At the unperturbed threshold $b=q$, \eqref{eq:neg-B1}--\eqref{eq:neg-D1} give the three identities that drive the construction:
\begin{equation}\label{eq:local-general-q}
 \mathbf B_{1,q}(q)=\frac{qT_q}{4}e_1,
 \qquad
 \mathbf D_{1,q}(q)=0,
 \qquad
 \mathbf D_{1,q}'(q)=qD_qe_1.
\end{equation}
Here the prime denotes componentwise differentiation with respect to the scalar threshold $b$.

The exact cap threshold $b=b(\theta,\varepsilon)$ is determined by
\[
 G(b)+\varepsilon\mathbf c_2(\theta)^T\mathbf B_{1,q}(b)
 +\varepsilon^2\mathbf c_4(\theta)^T\mathbf B_{2,q}(b)=p.
\]
Differentiating this fixed-mass equation at $\varepsilon=0$, where $b=q$, yields
\[
 b'(0)
 =\frac{qT_q}{4\phi(q)}\cos2\theta
 =\frac{q\kappa_q}{4}\cos2\theta.
\]
Thus, uniformly in $\theta$,
\[
b(\theta,\varepsilon)
 =q+\varepsilon\frac{q\kappa_q}{4}\cos2\theta+O_q(\varepsilon^2).
\]

Because the median stays central, the exact defect is
\[
 \Delta_{\mu_{\varepsilon,q},p}(\theta)
 =\varepsilon\mathbf c_2(\theta)^T\mathbf D_{1,q}(b)
 +\varepsilon^2\mathbf c_4(\theta)^T\mathbf D_{2,q}(b).
\]
Since $b-q=O(\varepsilon)$, the first defect coefficient must be expanded through first order in $b-q$, while the second needs only its value at $q$.  Because $\mathbf D_{1,q}(q)=0$,
\[
 \mathbf D_{1,q}(b)
 =\mathbf D'_{1,q}(q)(b-q)+O\bigl((b-q)^2\bigr),
 \qquad
 \mathbf D_{2,q}(b)=\mathbf D_{2,q}(q)+O(b-q).
\]
Using
\[
 b-q=\varepsilon\frac{q\kappa_q}{4}\cos2\theta+O_q(\varepsilon^2)
\]
together with \eqref{eq:local-general-q} gives
\begin{align*}
 \Delta_{\mu_{\varepsilon,q},p}(\theta)
 &=K_q\varepsilon^2
 \left(\frac14\cos^2 2\theta-\frac18\cos4\theta\right)
 +O_q(\varepsilon^3)\\
 &=\frac{K_q}{8}\varepsilon^2+O_q(\varepsilon^3),
\end{align*}
where the second equality is $2\cos^2u-\cos2u=1$.  This proves the uniform asymptotic \eqref{eq:neg-asymptotic}.

It remains only to justify the analytic properties of the density.  Both perturbing polynomials have nonzero angular degree, so their integrals against the radial density $g_1$ vanish. Write
\[
 \rho_{\varepsilon,q}=g_0(1+\varepsilon h_{1,q}+\varepsilon^2h_{2,q}).
\]
For fixed $q$, each $h_{j,q}$ and its derivatives through order two is a polynomial times $e^{-r^2/2}$, hence bounded.  Therefore
\[
 \varepsilon h_{1,q}+\varepsilon^2h_{2,q}\longrightarrow0
 \quad\text{uniformly in }C^2(\R^2).
\]
For all sufficiently small $\varepsilon$ the density is strictly positive, and
\[
 \nabla^2\log\rho_{\varepsilon,q}\longrightarrow-I
 \quad\text{uniformly on }\R^2.
\]
Thus, for sufficiently small \(\varepsilon\), the Hessian of \(-\log \rho_{\varepsilon,q}\) remains uniformly positive definite, and hence \(\rho_{\varepsilon,q}\) is strongly log-concave.  Finally, since $K_q>0$ and the remainder in \eqref{eq:neg-asymptotic} is uniform in $\theta$, decreasing $\varepsilon_0(q)$ once more gives \eqref{eq:neg-gap}.
\end{proof}

In the proof above, note that $K_q\to0$ as $q\downarrow0$ and as $q\to\infty$, corresponding respectively to $t\uparrow1/4$ and $t\downarrow0$. At the endpoints $t=0$ and $t=1/4$, this is consistent with the existence of perpendicular equipartitions.

\section{Paired cuts and strict monotonicity}\label{sec:monotonicity}

We now study the full zero set of the imbalance \eqref{eq:master-imbalance}, rather than its restriction to perpendicular directions. For the differential arguments below, call a probability measure \(\mu\) \emph{regular} if
$$ \dd\mu(x)=\rho(x)\,\dd A(x), $$
where \(\rho\) is \(C^1\), strictly positive on \(\R^2\), and decays sufficiently rapidly that the line integrals and first line moments used below are finite. In particular, the median and \(p\)-cap thresholds of Section~\ref{sec:variation} are well defined and unique. A zero of \(\Phi_{\mu,p}\) will be called a \emph{paired cut}: equivalently, the median line bisects the corresponding \(p\)-cap.

Existence of some paired cut already follows from the planar ham-sandwich theorem. Indeed, fix a \(p\)-cap \(C\) and bisect simultaneously the measures \(\mu|_C\) and \(\mu|_{\R^2\setminus C}\). The resulting four regions have masses
$$ \frac p2,\quad \frac p2,\quad \frac{1-p}{2},\quad \frac{1-p}{2}, $$
so the bisecting line is also a median of \(\mu\). We need more: for each fixed median direction we will obtain a unique paired direction, depending continuously on it, together with control of how the imbalance changes as that direction rotates.

Fix \(u\in S^1\) and write
$$ v(\varphi)=\mathcal R_\varphi u,\qquad 0\le\varphi\le\pi, $$
with counterclockwise rotation positive, as before. At the endpoints,
$$ \Phi_{\mu,p}(u,u)=\frac p2,\qquad \Phi_{\mu,p}(u,-u)=-\frac p2. $$
The following lemma gives the required monotonicity.

\begin{lemma}\label{lem:monotonicity}
For fixed $u$, the function $\varphi\mapsto\Phi_{\mu,p}(u,\mathcal R_\varphi u)$ is strictly decreasing on $(0,\pi)$. More specifically, at an interior angle let the boundary of the $p$-cap meet the median at $P$, and parameterize the cap boundary by $x(s)=P+se$, choosing the unit tangent $e$ so that $s>0$ lies in $H_u^+$.  Put
\[
 w(s)=\rho(P+se),\quad
 W_\pm=\int_{\R_\pm}w(s)\dd s,\quad
 M_\pm=\int_{\R_\pm}s\,w(s)\dd s.
\]
Then
\[
\frac{\dd}{\dd\varphi}\Phi_{\mu,p}(u,\mathcal R_\varphi u)
 =-\frac{W_-M_+-W_+M_-}{W_++W_-}<0.
\]
\end{lemma}

\begin{proof}
For a real threshold $q$ define
\begin{align*}
 \mathcal A(\varphi,q)
 &=\mu\{x:\langle x,v(\varphi)\rangle\ge q\},\\
 \mathcal B(\varphi,q)
 &=\mu\bigl(H_u^+\cap\{x:\langle x,v(\varphi)\rangle\ge q\}\bigr).
\end{align*}
The actual cap threshold $q=q(\varphi)$ is determined by
\[
 \mathcal A(\varphi,q(\varphi))=p,
\]
and
\[
 \Phi_{\mu,p}(u,v(\varphi))
 =\mathcal B(\varphi,q(\varphi))-\frac p2.
\]
Thus \eqref{eq:variation-constraint}--\eqref{eq:variation-chain} give
\begin{equation}\label{eq:monotonicity-chain}
 q'=-\frac{\mathcal A_\varphi}{\mathcal A_q},
 \qquad
 \Phi'=\mathcal B_\varphi-\frac{\mathcal B_q\mathcal A_\varphi}{\mathcal A_q}.
\end{equation}

It remains to compute four partial derivatives.  Along the cap boundary write $x=P+se$ as in the statement.  Choose $e$ with the same orientation as above; then $v'(\varphi)=-e$.  Hence
\[
 \langle P+se,v'\rangle=c-s,
 \qquad c:=\langle P,v'\rangle.
\]
Differentiating a half-space integral with respect to its threshold and its normal gives
\begin{align*}
 \mathcal A_q&=-(W_++W_-),\\
 \mathcal A_\varphi&=c(W_++W_-)-(M_++M_-),\\
 \mathcal B_q&=-W_+,\\
 \mathcal B_\varphi&=cW_+-M_+.
\end{align*}

The first equation in \eqref{eq:monotonicity-chain} therefore reads
\[
 q'=c-\frac{M_++M_-}{W_++W_-}.
\]
Substitution into the second gives
\begin{align*}
 \Phi'
 &=cW_+-M_+-W_+
 \left(c-\frac{M_++M_-}{W_++W_-}\right)\\
 &=-\frac{W_-M_+-W_+M_-}{W_++W_-}.
\end{align*}
Since $W_\pm>0$, $M_+>0$, and $M_-<0$, this is strictly negative.
\end{proof}

The point of the calculation is the cancellation of $c$: the compensating translation disappears from the derivative of the imbalance.

\begin{theorem}\label{thm:pairedmap}
For every regular $\mu$, every $0<p\le1/2$, and every $u\in S^1$, there is a unique direction $F_p(u)$ in the positively oriented open semicircle from $u$ to $-u$ such that
\[
 \Phi_{\mu,p}(u,F_p(u))=0.
\]
The map $F_p:S^1\to S^1$ is continuous and has degree one.
\end{theorem}

\begin{proof}
For fixed $u$, the endpoint values and Lemma~\ref{lem:monotonicity} give a unique zero.

For continuity, let $u_n\to u$ and set $v_n=F_p(u_n)$.  Every subsequence has a further convergent subsequence $v_{n_k}\to v_*$.  Continuity of $\Phi_{\mu,p}$ gives $\Phi_{\mu,p}(u,v_*)=0$; the endpoint signs exclude $v_*=\pm u$, and uniqueness gives $v_*=F_p(u)$.  Hence the whole sequence converges.  Writing
\[
 F_p(u_\theta)=u_{\theta+\gamma_p(\theta)},\qquad 0<\gamma_p(\theta)<\pi,
\]
produces a continuous $2\pi$-periodic angular displacement.  The lift $\theta\mapsto\theta+\gamma_p(\theta)$ gains $2\pi$ in one turn, so $F_p$ has degree one.
\end{proof}

The relation with the orthogonality defect of Section~\ref{sec:variation} is now immediate.  For $u=u_\theta$, strict monotonicity says that
\[
 \varphi\longmapsto\Phi_{\mu,p}(u_\theta,\mathcal R_\varphi u_\theta)
\]
is strictly decreasing and vanishes precisely at $\varphi=\gamma_p(\theta)$.  Evaluating it at $\varphi=\pi/2$ gives
\begin{equation}\label{eq:defect-bridge}
 \sgn\Delta_{\mu,p}(u_\theta)
 =\sgn\!\left(\gamma_p(\theta)-\frac\pi2\right)
 =-\sgn\!\bigl(\cot\gamma_p(\theta)\bigr),
\end{equation}
whenever these quantities are nonzero.  Thus the scalar defect used in the counterexample and the angular displacement of the paired directions record the same one-sided failure of orthogonality.  In particular, for every interior target, \eqref{eq:neg-gap} says that the paired cycle of the corresponding Gaussian counterexample lies entirely on the $\gamma_p>\pi/2$ side of the right-angle locus.

\section{The partition action and its affine sign law}
\label{sec:chirality}

Recall from Theorem~\ref{thm:pairedmap} that the paired direction
can be written as
\[
 F_p(u_\theta)=u_{\theta+\gamma_p(\theta)},
 \qquad 0<\gamma_p(\theta)<\pi,
\]
where $\gamma_p$ is continuous and $2\pi$-periodic. A perpendicular
paired cut exists precisely when $\gamma_p(\theta)=\pi/2$ for
some $\theta$. We introduce an integral whose vanishing forces
this equality and whose behavior under affine maps can be
computed by an elementary change of variable.

\begin{definition}\label{def:J}
For a regular measure $\mu$, define its \emph{partition action} by
\begin{equation}\label{eq:J}
 J_p(\mu):=\frac12\int_0^{2\pi}
 \cot\gamma_p(\theta)\,\dd\theta.
\end{equation}
The factor $1/2$ is a normalization.
\end{definition}

If no desired partition exists, 
$\cot\gamma_p$ never vanishes and the integral is nonzero. Conversely $J_p(\mu)=0$ implies that a perpendicular paired cut exists. The following calculation explains the choice of the cotangent
in \eqref{eq:J}, but further intuition is developed in Section \ref{sec:discussion}.

\begin{lemma}\label{lem:angular-change}
Let $T$ be an invertible real $2\times2$ matrix, and choose a
continuous real-valued angle $\alpha(\theta)$ such that
\[
 \frac{Tu_\theta}{\|Tu_\theta\|}=u_{\alpha(\theta)}.
\]
Then
\[
 \alpha'(\theta)=\frac{\det T}{\|Tu_\theta\|^2},
\]
and, whenever $\varphi-\theta\notin\pi\mathbb Z$,
\begin{equation}\label{eq:angular-change}
 \cot\bigl(\alpha(\varphi)-\alpha(\theta)\bigr)\alpha'(\theta)
 =
 \cot(\varphi-\theta)
 +\frac{\dd}{\dd\theta}\log\|Tu_\theta\|.
\end{equation}
\end{lemma}

\begin{proof}
Taking determinants after differentiating
$Tu_\theta=\|Tu_\theta\|u_{\alpha(\theta)}$ gives the first
identity. For the second, put $\gamma=\varphi-\theta$ and use
$u_\varphi=\cos\gamma\,u_\theta+\sin\gamma\,u_\theta^\perp$:
\[
 \cot\bigl(\alpha(\varphi)-\alpha(\theta)\bigr)\alpha'(\theta)
 =
 \frac{\langle Tu_\theta,Tu_\varphi\rangle}
      {\|Tu_\theta\|^2\sin\gamma}
 =
 \cot\gamma+\frac{\dd}{\dd\theta}\log\|Tu_\theta\|.
\]
\end{proof}

Let $A(x)=Lx+b$ be an invertible affine map, and let
$\widetilde\mu$ be the measure obtained by applying $A$ to
$\mu$, so that $\widetilde\mu(A(E))=\mu(E)$ 
for every Borel set $E$. A half-plane with normal $u$ is
carried to a half-plane with normal
\[
 a_L(u)=\frac{L^{-T}u}{\|L^{-T}u\|},
 \qquad L^{-T}=(L^{-1})^T.
\]
Thus median half-planes, $p$-caps, and paired intersections
are carried to objects satisfying the same mass conditions
for $\widetilde\mu$.

\begin{proposition}\label{prop:affineJ}
For every regular $\mu$ and every invertible affine map
$A(x)=Lx+b$, the transformed measure $\widetilde\mu$ satisfies
\[
 J_p(\widetilde\mu)=\sgn(\det L)J_p(\mu).
\]
\end{proposition}

\begin{proof}
Put $T=L^{-T}$ and
$s=\sgn(\det L)=\sgn(\det T)$. Choose the angle $\alpha$ from
Lemma~\ref{lem:angular-change}. The map
$u\mapsto Tu/\|Tu\|$ is a bijection of the direction circle,
and $\alpha'$ has sign $s$. Thus
\[
 \alpha(\theta+2\pi)=\alpha(\theta)+2\pi s.
\]

For each $\theta$, put
$\varphi=\theta+\gamma_p(\theta)$. The transformed cap has
normal $u_{\alpha(\varphi)}$. If $s=1$, its paired median
normal is $u_{\alpha(\theta)}$, and the cap normal lies in
the positive semicircle from this median normal.

If $s=-1$, the transformed cap normal instead lies in the
negative semicircle from $u_{\alpha(\theta)}$. We therefore
reverse the median normal, retaining the cap normal.
This still gives a paired cut: replacing the median
half-plane by its complement changes the intersection
mass from $p/2$ to $p-p/2=p/2$.

Accordingly, define
\[
 \tau(\theta)=
 \begin{cases}
  \alpha(\theta),&s=1,\\
  \alpha(\theta)+\pi,&s=-1.
 \end{cases}
\]
By uniqueness of the paired direction for $\widetilde\mu$,
and the $\pi$-periodicity of the cotangent,
\[
 \cot\gamma_{p,\widetilde\mu}(\tau(\theta))
 =
 \cot\bigl(\alpha(\varphi)-\alpha(\theta)\bigr),
\]
where $\gamma_{p,\widetilde\mu}$ denotes the angular displacement
for $\widetilde\mu$.

As $\theta$ runs from $0$ to $2\pi$, the angle $\tau(\theta)$
traverses the circle once, with orientation $s$, and
$\tau'=\alpha'$. Changing variables in \eqref{eq:J} and
applying \eqref{eq:angular-change} gives
\begin{align*}
 J_p(\widetilde\mu)
 &=
 \frac{s}{2}\int_0^{2\pi}
 \cot\gamma_{p,\widetilde\mu}(\tau(\theta))
 \alpha'(\theta)\,\dd\theta\\
 &=
 \frac{s}{2}\int_0^{2\pi}
 \left(
 \cot\gamma_p(\theta)
 +\frac{\dd}{\dd\theta}\log\|Tu_\theta\|
 \right)\dd\theta\\
 &=sJ_p(\mu).
\end{align*}
The last equality follows from the periodicity of
$\|Tu_\theta\|$.
\end{proof}

We can now obtain the positive result from affine reflection
symmetry.

\begin{theorem}\label{thm:regular}
Let $\mu$ be regular and invariant under an affine reflection.
For every $0\leq t\leq1/4$, there are two Euclidean-perpendicular
lines whose cyclic region masses are
\[
 t,\quad t,\quad \frac12-t,\quad\frac12-t.
\]
\end{theorem}

\begin{proof}
The endpoints $t=0$ and $t=1/4$ were addressed in
Section~\ref{sec:variation}. For $0<t<1/4$, put $p=2t$.
Let $R$ be the affine reflection preserving $\mu$.
Applying $R$ leaves the measure unchanged, whereas
Proposition~\ref{prop:affineJ} shows that it reverses
the sign of the partition action. Consequently, $J_p(\mu)=-J_p(\mu),$ so $J_p(\mu)=0$, and thus there is
a perpendicular paired cut, as desired.
\end{proof}

Theorem~\ref{thm:mainintro} follows from
Theorem~\ref{thm:regular} by the symmetry preserving
approximation in Appendix~\ref{app:line-null}.

\section{Convex-body consequences}\label{sec:convex}

\begin{corollary}\label{cor:body}
Let $K\subset\R^2$ be a convex body with an orientation-reversing affine automorphism.  Then for every $0\leq t\le1/4$ there are two Euclidean-perpendicular lines cutting $K$ cyclically into areas
\[
 t,\quad t,\quad\frac12-t,\quad\frac12-t,
\]
after normalizing $\operatorname{area}(K)=1$.
\end{corollary}

\begin{proof}
Let $E$ be the John ellipsoid of $K$, which is unique \cite[Theorem~10.12.2, p.~588]{Schneider2014}.  Every affine automorphism of $K$ therefore preserves $E$ and fixes its center.  After translating this center to the origin and conjugating by a linear map carrying $E$ to the unit disk, the automorphism becomes orthogonal.  In the plane every orientation-reversing orthogonal transformation is a reflection, hence has order two.  Thus the original automorphism is an affine reflection, and Theorem~\ref{thm:mainintro} applies to uniform area on $K$.
\end{proof}

\begin{corollary}\label{cor:triangle}
Every triangle satisfies Gr\"unbaum's uneven orthogonal partition conjecture for every target.
\end{corollary}

\begin{proof}
Every triangle has an orientation-reversing affine automorphism: interchange two vertices and fix the third.  Corollary~\ref{cor:body} applies.
\end{proof}

\begin{corollary}\label{cor:trapezoid}
Every convex trapezoid satisfies the Gr\"unbaum uneven orthogonal partition conjecture for every target.
\end{corollary}

\begin{proof}
After an affine normalization, write its vertices as
\[
 A=(0,0),\quad B=(1,0),\quad D=(u,1),\quad C=(v,1).
\]
Then
\[
 R(x,y)=\bigl(1+(u+v-1)y-x,\,y\bigr)
\]
interchanges $A$ with $B$ and $D$ with $C$, has determinant $-1$, and satisfies $R^2=\mathrm{id}$.  Corollary~\ref{cor:body} applies.
\end{proof}

\section{Discussion and further directions}\label{sec:discussion}

\subsection*{Gaussian perturbations and negative results}

Many counterexamples in mass-partition problems are built from measures concentrated near configurations on which the cutting objects have little freedom.  Avis's moment-curve construction is a classical example \cite{Avis1984}.  B\'ar\'any and Matou\v{s}ek use related elementary constructions for partitions by fans: atomic or segment-supported measures force rays to satisfy rigid incidence conditions, and in one example a small perturbation moves a mass-determined ray and destroys the required concurrency \cite{BaranyMatousek2001}.  Such constructions expose the geometric obstruction very directly.

Sober\'on's four-dimensional counterexample introduced a particularly flexible source of examples: small perturbations of a Gaussian \cite{Soberon2026}. Our construction follows this viewpoint in the plane. Polynomial perturbations can be organized by angular frequency and by the order at which they affect the defect, while the required calculations reduce largely to truncated Gaussian moments. This allows us both to construct smooth counterexamples and to track systematically how mass-defined boundaries move under perturbation. Subsequently, Frick \cite{Frick2026} disproved the Ramos conjecture for two hyperplanes using families of unperturbed Gaussian measures, whose centers and covariance forms encode an obstruction to simultaneous equipartition. His construction connects Gaussian mass partitions to phase retrieval, providing another way to turn their explicit structure into algebraic obstructions.

\subsection*{Metric invariants and topological methods}

Topological methods have been extraordinarily successful in mass-partition problems, particularly when symmetries of the configuration space can be encoded in an equivariant test map; see the survey of Roldán-Pensado and Soberón \cite{RoldanSoberon2022} and the work of Blagojević, Frick, Haase and Ziegler \cite{BlagojevicFrickHaaseZiegler2016,BlagojevicFrickHaaseZiegler2018}. Particularly close to the present problem, Blagojević and Dimitrijević Blagojević obtained a related spherical \(4\)-fan theorem by equivariant methods \cite{BlagojevicDimitrijevic2013}, while McGinnis and Zerbib recently highlighted Grünbaum's conjecture as a natural open problem for KKM-type methods \cite{McGinnisZerbib2024}.

The results here suggests that in constrained partition problems it may be useful to look, alongside topological obstructions, for metric or projective invariants whose sign detects the desired geometric condition.  Symmetries that reverse such a sign can then force the invariant to vanish.  The convex-geometric argument of Arocha, Jer\'onimo-Castro, Montejano and Rold\'an-Pensado \cite{ArochaEtAl2010}, which obtains opposite defect signs from geometric information about the body, fits broadly into the same philosophy.

\subsection*{Liouville geometry of the paired cycle}

The partition action also has a hyperbolic interpretation.
As $\theta$ runs from $0$ to $2\pi$, the pairs
$([u_\theta],[F_p(u_\theta)])$ trace a closed oriented cycle
$C_p(\mu)$ in
\[
 \mathcal C=\RP^1\times\RP^1\setminus\Delta,
\]
where $\Delta$ is the diagonal and the cycle is understood
with the orientation and multiplicity induced by this
parametrization. Although $F_p$ need not descend to a map
on $\RP^1$, this cycle is well defined. The space
$\mathcal C$ can be identified with the space of oriented
hyperbolic geodesics by recording their ordered endpoints
at infinity; see Bracho \cite{Bracho2009}.

In angular coordinates $(\theta,\varphi)$ on $\mathcal C$,
a standard normalization of the Liouville area form is
\[
 \omega=
 \frac{\dd\theta\wedge\dd\varphi}{\sin^2(\varphi-\theta)};
\]
in slope coordinates this becomes
$\dd x\wedge\dd y/(y-x)^2$
\cite[\S2 and Appendix~A2]{Bonahon1988}.
The one-form
\[
 \lambda=\cot(\varphi-\theta)\,\dd\theta
\]
satisfies $\dd\lambda=\omega$, and the partition action is precisely
\[
 J_p(\mu)=\frac12\oint_{C_p(\mu)}\lambda.
\]
Thus the logarithmic derivative in
\eqref{eq:angular-change} is an exact differential whose
integral around the cycle vanishes, providing a geometric
interpretation of the elementary calculation in
Section~\ref{sec:chirality}. It would be interesting to
understand whether cycles arising from paired cuts satisfy
further restrictions from this hyperbolic viewpoint.

\subsection*{Remaining questions}

Theorem~\ref{thm:counterintro} settles B\'ar\'any's general-measure question, but Gr\"unbaum's convex-body problem remains open. The corollaries in Section \ref{sec:convex} deal with triangles and trapezoids. The next natural test case is that of general convex quadrilaterals, which remains open. We believe this family could yield new insights on the problem.

\begin{problem}\label{prob:quadri}
Does every convex quadrilateral satisfy Gr\"unbaum's conjecture for every target?
\end{problem}

The positive argument naturally raises the question of which paired cycles can occur. In view of the hyperbolic interpretation above, this may also be regarded as a realization problem for curves in the space of oriented hyperbolic geodesics.

\begin{problem}\label{q:realization}
Which closed oriented curves in
\[
 \RP^1\times\RP^1\setminus\Delta
\]
arise as paired cycles $C_p(\mu)$ of regular planar measures?  What additional restrictions arise for uniform measures on convex bodies?
\end{problem}

There is also a basic question about the set of admissible targets for a fixed measure \(\mu\). Our counterexamples are constructed separately for each \(t\in(0,1/4)\), so the possible behavior of the target set for a single measure remains open.

\begin{problem}\label{q:targets}
For a fixed line-null probability measure $\mu$, let
\[
 T(\mu)=
 \left\{
 t\in[0,1/4]:
 \text{$\mu$ admits the prescribed orthogonal partition at target $t$}
 \right\}.
\]
What sets can occur as $T(\mu)$?  What additional structure does $T(\mu)$ have for regular measures, or for uniform measures on convex bodies?
\end{problem}

The finite counterexample also leaves a basic quantitative problem open.  Maldonado--Rold\'an-Pensado prove that $k=1$ is always solvable, while Theorem~\ref{thm:discrete-counterexample} gives a failure at $k=8$.

\begin{problem}\label{q:smaller-discrete}
What is the smallest integer $k$ for which the discrete orthogonal partition problem admits a counterexample?  In particular, does a counterexample exist for some $2\le k\le7$?
\end{problem}

\appendix

\section{From regular to line-null and discrete measures}\label{app:line-null}

For completeness, we record the approximation arguments used to pass from the regular theorem to line-null measures and to the discrete corollary.  The construction is standard; the only point requiring care here is that the affine reflection symmetry and the moving partition boundaries must survive the limit.

\subsection*{Symmetry-preserving regularization}
Let $R$ be the orientation-reversing affine involution preserving $\mu$.  Translate a fixed point of $R$ to the origin and write $R(x)=Lx$, where $L^2=I$ and $\det L=-1$.  Choose an $L$-invariant positive-definite quadratic form, for instance
\[
 Q(z)=|z|^2+|Lz|^2.
\]
Let $B_Q(r)=\{Q(z)\le r^2\}$.  These balls are $R$-invariant.  Restrict $\mu$ to $B_Q(r_n)$, renormalize, and convolve with a positive $Q$-radial Gaussian kernel of scale $\varepsilon_n$, where $r_n\to\infty$ and $\varepsilon_n\downarrow0$.  The resulting probabilities $\mu_n$ are $R$-invariant, have smooth strictly positive densities with rapid decay, and satisfy
\[
 \mu_n\Rightarrow\mu.
\]
Thus each $\mu_n$ is regular in the sense of Section~\ref{sec:monotonicity}.  This is a standard truncation-and-mollification argument, see Folland \cite[Section~8.2]{Folland1999}.  The $L$-invariant metric is used only to preserve the affine reflection exactly during the regularization.

\subsection*{Compactness of the cutting lines}
Apply Theorem~\ref{thm:regular} to $\mu_n$.  Label the two perpendicular lines so that the first is a median line and the second bounds the side of mass $p$.  Write them as
\[
 \langle x,u_n\rangle=a_n,
 \qquad
 \langle x,v_n\rangle=b_n,
\]
with $u_n,v_n\in S^1$.  Passing to a subsequence, the directions converge.  The offsets are bounded as well.  Indeed, weak convergence implies uniform tightness of $(\mu_n)$: for every small $\varepsilon>0$ there is a Euclidean ball $B$ with $\mu_n(B)>1-\varepsilon$ for all sufficiently large $n$.  If $\varepsilon<\min\{p,1/2\}$, a median line cannot escape beyond $B$, because then one of its half-planes would have mass less than $\varepsilon$; similarly, the boundary of a side of mass $p$ cannot escape in either direction, since that side would then have mass either less than $\varepsilon$ or greater than $1-\varepsilon$.  Hence, after another subsequence,
\[
 (u_n,a_n,v_n,b_n)\longrightarrow(u,a,v,b).
\]
The limiting lines remain perpendicular because $u_n\cdot v_n=0$ for every $n$.

For a fixed sector, convergence of its mass is the continuity-set clause of the Portmanteau theorem; see Billingsley \cite[Theorem~2.1]{Billingsley1999}.  The following elementary statement is the varying-boundary version needed here.

\begin{lemma}[Moving sectors]\label{lem:movingsectors}
Suppose $\nu_n\Rightarrow\nu$.  Let $S_n$ be an open sector bounded by two lines whose defining normals and offsets converge to those of an open sector $S$.  If $\nu$ gives zero mass to the two limiting boundary lines, then
\[
 \nu_n(S_n)\longrightarrow\nu(S).
\]
\end{lemma}

\begin{proof}
Fix $\eta>0$.  By tightness, choose a large ball $B$ carrying all but $\eta$ of the mass of $\nu$ and, for all sufficiently large $n$, all but $2\eta$ of the mass of $\nu_n$.  Since the two limiting lines have $\nu$-mass zero, the $\nu$-mass of a $\delta$-neighborhood of their union tends to zero with $\delta$.  On the fixed ball $B$, convergence of normals and offsets implies that for large $n$ the symmetric difference $S_n\triangle S$ is contained in such a $\delta$-neighborhood.  Inner and outer approximations of $S\cap B$ by closed and open sets, together with weak convergence, then give
\[
 \limsup_n|\nu_n(S_n)-\nu(S)|\le C\eta
\]
for an absolute constant $C$.  Letting $\eta\downarrow0$ proves the claim.
\end{proof}

Apply Lemma~\ref{lem:movingsectors} to each of the four sectors of the convergent perpendicular pair.  Their limiting boundaries lie in the two limiting lines, both of $\mu$-mass zero.  Hence the four masses pass to the limit and equal
\[
 \frac p2,\quad\frac p2,\quad\frac{1-p}{2},\quad\frac{1-p}{2}.
\]
Taking $p=2t$ proves Theorem~\ref{thm:mainintro}.

\subsection*{The discrete corollary}
We now prove Corollary~\ref{cor:discrete}.

\begin{proof}
Give each point mass $1/n$ and replace every atom by a small positive Gaussian bump with respect to an inner product invariant under the linear part of the affine involution.  Using the same bump on every orbit gives a regular invariant probability measure $\mu_\varepsilon$.  Apply Theorem~\ref{thm:regular} with $p=2k/n$ and let $\varepsilon\downarrow0$.  After taking a convergent subsequence of the perpendicular pairs, every point lying strictly inside a limiting open sector eventually contributes asymptotically its full mass $1/n$ to the corresponding smooth sector.  Since the smooth sector masses are exactly $k/n,k/n,(n-2k)/(2n),(n-2k)/(2n)$, the stated upper bounds follow.
\end{proof}

When no point lies on either cutting line and the four target numbers are integers, the upper bounds sum to $n$ and therefore become equalities.  In general boundary points may be unavoidable.

\section{An explicit finite counterexample}\label{app:discrete-counterexample}
We finish with the computer-assisted finite result stated in Theorem~\ref{thm:discrete-counterexample}.  A \emph{strict} partition uses two perpendicular lines avoiding the points.  A \emph{weak} partition may place points on either line and assign each such point to any of its incident regions.

Let $|P|=2m$ and seek counts $(k,k,m-k,m-k)$.  For a nonzero cap normal $n$, put $u=Jn=(-n_y,n_x)$.  Let $A=\Top_m(u)$ and $B=\Top_{2k}(n)$ be the index sets of the largest indicated projections, allowing every admissible choice at a tied cutoff.  Write $q(n;A,B)=|A\cap B|$. The four cyclic counts determined by these two threshold half-planes are
\[
 q,\quad m-q,\quad m-2k+q,\quad 2k-q.
\]
Consequently a weak target partition exists exactly when there is a choice such that 
\begin{equation}\label{eq:discrete-rank}
 q(n;A,B)=k.
\end{equation}
This reduces the search over pairs of perpendicular lines to a finite sweep over the direction \(n\).

In our counterexample, $m=48$ and $k=8$. Using the points $p_i$ from Table~\ref{tab:discrete-points}, the point set is $$P=\{p_1,\ldots,p_{48},-p_1,\ldots,-p_{48}\}.$$

\begin{table}[ht]
\centering\small
\renewcommand{\arraystretch}{1.08}
\begin{tabular}{r rr@{\qquad}r rr}
\toprule
$i$ & $x_i$ & $y_i$ & $i$ & $x_i$ & $y_i$\\
\midrule
1 & 6501822 & -2132732 & 25 & 802940 & 10847176 \\
2 & 15605488 & 5179146 & 26 & 3156889 & -2804411 \\
3 & 12164442 & 4079437 & 27 & -1513310 & 8052343 \\
4 & 11196842 & 3257208 & 28 & -1270507 & 8035601 \\
5 & 2720115 & -1016910 & 29 & -15239401 & 22823153 \\
6 & 14742722 & 5251779 & 30 & 3404526 & 3376251 \\
7 & 12027311 & 4754946 & 31 & -8703705 & 3450614 \\
8 & 5567252 & 6636475 & 32 & -3796526 & 7686574 \\
9 & 13019739 & 6905511 & 33 & -14003196 & 22067042 \\
10 & 16430000 & 16782244 & 34 & -11395404 & 12761290 \\
11 & 6290597 & 5509041 & 35 & -5828330 & 7246273 \\
12 & 5219998 & 6968694 & 36 & -14174109 & 18264971 \\
13 & 2016438 & 2205060 & 37 & -16952370 & 23820851 \\
14 & 5219033 & 6950426 & 38 & -9691753 & 6788722 \\
15 & 13124961 & 20240705 & 39 & -6093427 & 5843902 \\
16 & 5798488 & 6042258 & 40 & -8209011 & 4191810 \\
17 & 5263233 & 6936155 & 41 & -9871547 & 6272237 \\
18 & 1756625 & 10727867 & 42 & -11267533 & -795180 \\
19 & 8477144 & 15858722 & 43 & -8805970 & 3228616 \\
20 & 1728935 & 10746669 & 44 & -9671078 & 1515998 \\
21 & 2851789 & 9062043 & 45 & -8816653 & 3213163 \\
22 & -2853968 & 12408044 & 46 & -8753576 & 3324682 \\
23 & -1439883 & 8032057 & 47 & -8450033 & 3913033 \\
24 & 1836987 & 1695724 & 48 & -6832402 & 5424922 \\
\bottomrule
\end{tabular}
\caption{$48$ points from the $96$-point counterexample; the other points are their negatives.}
\label{tab:discrete-points}
\end{table}

The configuration was found by a floating-point simulated-annealing search tailored to centrally symmetric sets. The search evaluates the energy
$$E(P)=\sum_{\text{chambers}} \max\{0,9-q(n;A,B)\}^2.$$
Random local polar perturbations of single representatives are accepted or rejected according to an annealing schedule. When the energy reached \(0\), the floating-point coordinates were rounded to integers and then certified exactly as follows.

For every difference $d=p_i-p_j\ne0$, changes in either relevant projection order can occur only when $n$ is parallel or perpendicular to $d$.  These finitely many critical rays cut the direction circle into $9{,}216$ open chambers.  Inside each chamber both projection orders are constant, so a single exact integer-arithmetic sample determines $q(n;A,B)$ throughout the chamber.  At a critical direction, all admissible assignments at the two tied cutoffs are checked exactly.

Two separately implemented verifiers show that $q(n;A,B)-8\in\{1,2\}$ 
for every open chamber and every admissible critical-direction assignment. Thus \eqref{eq:discrete-rank} never occurs.  Exact determinants also verify that the $96$ points are distinct and that no three are collinear.  This proves Theorem~\ref{thm:discrete-counterexample}.

\section*{Acknowledgements}
The author thanks Edgardo Rold\'an-Pensado for pointing him  to Gerardo Maldonado's thesis. This work was supported by UNAM-PAPIIT project IN119026. The author also acknowledges support from the UNAM DGAPA-PASPA program for a sabbatical stay at the Instituto de Matem\'aticas, Unidad Juriquilla, UNAM.

\section*{Software and computational verification}

Python was used to search for the discrete counterexample by
simulated annealing and to verify the resulting configuration.
The exploratory search used floating-point arithmetic and the final
integer configuration was certified using exact integer arithmetic,
as described in Appendix~\ref{app:discrete-counterexample}.
The search and verification scripts, together with a Jupyter
notebook and instructions for reproducing the computations, are
available at \texttt{https://github.com/leomtz/quadrupartitions}.

\section*{AI Declaration}

The research leading to this work involved an extended interaction with ChatGPT (OpenAI), using GPT-5.6 Sol. The model was used throughout the project for mathematical exploration, development and refinement of proof strategies, preliminary drafting, and computational work. The arguments, formulations, and computations in the paper emerged through repeated collaborative interaction, revision, and correction during this process. The final proofs, computations, and references have been checked by the author, who assumes full responsibility for the mathematical content of the paper, including any errors or bibliographic omissions.

\end{document}